\documentclass{amsart}

\usepackage{amsmath}
\usepackage{amsthm}
\usepackage{amsfonts}
\usepackage{amssymb}
\usepackage{graphics}
\usepackage{epsfig}
\usepackage{color}
\usepackage{verbatim}

\newcommand\diag{\operatorname{diag}}

\newcommand\supp{\operatorname{supp}}

\theoremstyle{plain}
  \newtheorem{theorem}{Theorem}

  \newtheorem{proposition}[theorem]{Proposition}
  
  \newtheorem{lemma}[theorem]{Lemma}
  \newtheorem{corollary}[theorem]{Corollary}

\theoremstyle{definition}

\date{5th June 2013}

\begin{document}

\title[Control of singular multipliers by maximal operators]{Optimal control of singular Fourier multipliers by maximal operators}
%Marcinkiewicz-type theorems for oscillatory multipliers
\author{Jonathan Bennett}
\address{School of Mathematics \\
The University of Birmingham \\
The Watson Building \\
Edgbaston \\
Birmingham \\
B15 2TT \\
United Kingdom}
%\curraddr{}
\email{J.Bennett@bham.ac.uk}
\thanks{Supported by ERC Starting Grant
307617.}

\begin{abstract}
We control a broad class of singular (or ``rough") Fourier multipliers by geometrically-defined maximal operators via general weighted $L^2(\mathbb{R})$ norm inequalities. The multipliers involved are related to those of Coifman--Rubio de Francia--Semmes, satisfying certain weak Marcinkiewicz-type conditions that permit highly oscillatory factors of the form $e^{i|\xi|^\alpha}$ for both $\alpha$ positive and negative. The maximal functions that arise are of some independent interest, involving fractional averages associated with tangential approach regions (related to those of Nagel and Stein), and more novel ``improper fractional averages" associated with ``escape" regions. Some applications are given to the theory of $L^p-L^q$ multipliers, oscillatory integrals and dispersive PDE, along with natural extensions to higher dimensions.
\end{abstract}

%\begin{abstract}We prove a weighted Marcinkiewicz-type theorem for singular (or ``rough") Fourier multiplier operators on $\mathbb{R}^n$. The two weights are related by fractional maximal operators %associated with both approach and ``escape" regions. Applications are given to the theory of $L^p-L^q$ multipliers and the regularity theory of linear dispersive PDE.
%\end{abstract}
\maketitle
\begin{center}
\textit{Dedicated to the memory of Adela Moyua, 1956--2013.}
\end{center}

\section{Introduction and statements of results}
%BRIEF DISCUSSION OF WEIGHTED NORM INEQUALITIES IN HARMONIC ANALYSIS, EMPHASISING OSCILLATORY.
Given a Fourier multiplier $m$, with corresponding convolution operator $T_m$, there has been
considerable interest in identifying, where possible, ``geometrically-defined" maximal operators
$\mathcal{M}$ for which a weighted $L^2$-norm inequality of the form
\begin{equation}\label{want}
\int_{\mathbb{R}^n}|T_mf|^2w\leq\int_{\mathbb{R}^n}|f|^2\mathcal{M}w
\end{equation}
holds for all admissible input functions $f$ and weight functions $w$.
This very general Fourier multiplier problem was made particularly explicit
in the 1970s in work of A. C\'ordoba and C. Fefferman \cite{CF}, following the emergence of fundamental connections between the theory of Fourier multipliers and elementary geometric notions such as curvature (see in particular \cite{Feff}, \cite{CordKak}, \cite{S}).
Such control of a multiplier $m$ by a maximal operator $\mathcal{M}$, combined with an elementary duality argument, reveals that for $p,q\geq 2$,
\begin{equation}\label{duality}
\|m\|_{p,q}:=\|T_m\|_{L^p-L^q}\leq \|\mathcal{M}\|_{L^{(q/2)'}-L^{(p/2)'}}^{1/2}.
\end{equation}
Thus it is of particular interest to identify an ``optimal" maximal operator $\mathcal{M}$ for which \eqref{want} holds, in the sense that \eqref{duality} permits optimal $L^p-L^q$ bounds for $\mathcal{M}$ to be transferred to optimal bounds for $T_m$.

There are a variety of results of this nature, although often formulated in terms of the convolution kernel rather than the multiplier.
For example, if $T$ denotes a Calder\'on--Zygmund singular integral operator on $\mathbb{R}^n$, such as the Hilbert transform on the line, C\'ordoba and Fefferman \cite{CF} (see also \cite{HMW}) showed that for each $s>1$ there is a constant $C_s<\infty$ for which
\begin{equation}\label{hilbertprim}
\int_{\mathbb{R}}|Tf|^2w\leq C_s\int_{\mathbb{R}}|f|^2(Mw^s)^{1/s},
\end{equation}
holds, where $M$ denotes the classical Hardy--Littlewood maximal operator. This result extends to weighted $L^p$ estimates for $1<p<\infty$; see \cite{CF}. The inequality \eqref{hilbertprim} may be viewed as a consequence of the classical theory of Muckenhoupt $A_p$ weights through the fundamental fact that if $(Mw^s)^{1/s}<\infty$ a.e. and $s>1$ then $(Mw^s)^{1/s}\in A_1\subset A_2$; see \cite{S1} and the references there. Of course, for any fixed $s>1$ the maximal operator $w\mapsto (Mw^s)^{1/s}$ in \eqref{hilbert} is not optimal since it fails to be $L^p$-bounded in the range $1<p\leq s$, while $T$ is bounded on $L^p$ for all $1<p<\infty$. More recently this was remedied by Wilson \cite{W}, who showed that\footnote{Throughout this paper we shall write $A\lesssim B$ if there exists a constant $c$ such that $A\leq cB$. In particular, this constant will always be independent of the input function $f$ and weight function $w$. The relations $A\gtrsim B$ and $A\sim B$ are defined similarly.}
\begin{equation}\label{hilbert}
\int_{\mathbb{R}}|Tf|^2w\lesssim\int_{\mathbb{R}}|f|^2M^3w,
\end{equation}
where $M^3=M\circ M\circ M$ denotes the 3-fold composition of $M$ with itself. As with \eqref{hilbertprim}, this useful result extends to weighted $L^p$ norms for $1<p<\infty$; see \cite{W}, \cite{P1}, \cite{RT}. There are numerous further results belonging to the considerable theory surrounding the $A_p$ weights; see for example \cite{GCRdF}, \cite{P2}, \cite{Hyt}, \cite{LSU}, \cite{HLP}, \cite{Lerner}.

In the setting of \emph{oscillatory integrals} the controlling maximal operators appear to acquire a much more interesting geometric nature, well beyond the scope of the classical $A_p$ theory.
%It is this feature which provides the motivation for this paper.
This is illustrated well by a compelling and seemingly very deep conjecture concerning the classical \emph{Bochner--Riesz multipliers}
$$
m_\delta(\xi)=\max\{(1-|\xi|^2)^\delta, 0\},
$$
where $\xi\in\mathbb{R}^d$ and $\delta\geq 0$. Of course, $m_0$ is simply the characteristic function of the unit ball in $\mathbb{R}^d$, allowing us to interpret $m_\delta$, for $\delta>0$, as a certain regularisation of this characteristic function. The classical Bochner--Riesz conjecture concerns the range of exponents $p$ for which $m_\delta$ is an $L^p$-multiplier.
In the 1970s A. C\'ordoba \cite{CordKak} and E. M. Stein \cite{S} raised the possibility that a weighted inequality of the form
\eqref{want} holds where $\mathcal{M}$ is some suitable variant of the Nikodym maximal operator
$$
\mathcal{N}_\delta w(x):=\sup_{T\ni x}\frac{1}{|T|}\int_T w;
$$
see also \cite{Feff}, \cite{Feff1}.
Here the supremum is taken over all cylindrical tubes of eccentricity less that $1/\delta$, containing the point $x$. This maximal operator $\mathcal{M}$ should be geometrically-defined (very much like $\mathcal{N}_\delta$) and its known/conjectured bounds should be similar to those of $\mathcal{N}_\delta$, thus essentially implying the full Bochner--Riesz conjecture via \eqref{duality}. \footnote{Similar weighted inequalities relating the Fourier restriction and Kakeya conjectures have also received some attention in the literature; see \cite{BCSV} for further discussion.} Such a result is rather straightforward for $d=1$ as it reduces to the aforementioned inequality for the Hilbert transform.
In higher dimensions this question is far from having a satisfactory answer already for $d=2$ (see \cite{Bo}, \cite{Christ}, \cite{CRS}, \cite{CS1}, \cite{CS2}, \cite{CSe}, \cite{BCSV}, \cite{DMOS}, \cite{LRS}, \cite{CR} for some related results). The associated convolution kernel
$$
K_\delta(x):=\mathcal{F}^{-1}m_\delta(x)=\frac{cJ_{d/2+\delta}(2\pi|x|)}{|x|^{\frac{d}{2}+\delta}}=c\frac{e^{2\pi i|\xi|}+e^{-2\pi i|\xi|}+o(1)}{|\xi|^{\frac{d+1}{2}+\delta}},$$
unlike the Hilbert kernel,
is (for $\delta$ sufficiently small) very far from being Lebesgue integrable. Here $J_\lambda$ denotes the Bessel function of order $\lambda$, making $K_\delta$ highly oscillatory.

%In addition to embarking on such difficult higher dimensional questions it would seem appropriate to reach a deeper understanding of the one dimensional situation for natural classes of oscillatory kernels.

In \cite{BH}, using arguments from \cite{BCSV} in the setting of Fourier extension operators, Harrison and the author gave nontrivial examples of such ``optimal" control of oscillatory kernels on the line by geometrically-defined maximal operators.
In particular, for integers $\ell\geq 3$, they showed that
\begin{equation}\label{SamJon}
\int_{\mathbb{R}}|e^{i(\cdot)^\ell}*f|^2w\lesssim\int_{\mathbb{R}}|f|^2M^4\mathcal{M}M^2w,
\end{equation}
where
$$\mathcal{M}w(x):=\sup_{(y,r)\in \Gamma(x)}\frac{1}{r^{\frac{1}{\ell-1}}}\int_{y-r}^{y+r}w,
$$
and
\begin{equation}\label{SamJonApp}\Gamma(x)=\{(y,r):0<r\leq 1;\;\;|x-y|\leq r^{-\frac{1}{\ell-1}}\}.\end{equation}
The maximal operator $\mathcal{M}$ here may be interpreted as a fractional Hardy--Littlewood maximal operator associated with an approach region
$\Gamma(x)$. This maximal operator is closely related to those studied by Nagel and Stein in \cite{NS}, although here tangential approach to infinite order is permitted. It is shown in \cite{BH} that $\mathcal{M}$ has a sharp bound on $L^{(\ell/2)'}$, which may be reconciled via \eqref{SamJon} with a sharp $L^\ell$ bound for convolution with $e^{ix^\ell}$. We note in passing that the factors of Hardy--Littlewood maximal operator appearing in \eqref{SamJon} are of secondary importance as $\mathcal{M}$ and $M^4\mathcal{M}M^2$ share the same $L^p-L^q$ mapping properties. This follows from the $L^p$-boundedness of $M$ for $1<p\leq\infty$.
%, although we refer the reader to \cite{CRS}, \cite{CS1}, \cite{CS2}, \cite{CSe}, \cite{BCSV}, \cite{LRS} \cite{CR} for some related results.

In this paper we seek an understanding of the ``map" $m\mapsto\mathcal{M}$, from Fourier multiplier to optimal controlling maximal operator, for which \eqref{want} holds.
As we shall see, an inequality of the form \eqref{want} does indeed hold for a wide class of multipliers $m$ and a surprisingly rich family of geometrically-defined maximal operators $\mathcal{M}$. This class of multipliers is sufficiently singular to apply to a variety of highly oscillatory convolution kernels, placing \eqref{SamJon} in a much broader context.
The maximal operators turn out to be fractional Hardy--Littlewood maximal operators associated with a diverse family of approach and ``escape" regions in the half-space. While such operators corresponding to approach regions have arisen before (see \cite{NS}, \cite{BCSV}, \cite{BH}), those associated with ``escape" regions appear to be quite novel, involving improper-fractional averages.

As is well known, in one dimension at least, the \emph{variation} of a multiplier can play a decisive role in determining its behaviour as an operator.
For example, if a multiplier $m$ is of bounded variation on the line, then it often satisfies the same norm inequalities as the Hilbert transform. This is a straightforward consequence of the elementary identity
\begin{equation}\label{marcrep}
T_m=\lim_{t\rightarrow -\infty}m(t)I+\frac{1}{2}\int_{\mathbb{R}}(I+iM_{-t}HM_t)dm(t).
\end{equation}
Here $I$ denotes the identity operator on $\mathbb{R}^d$, the modulation operator $M_t$ is given by $M_tf(x)=e^{-2\pi ixt}f(x)$, and $dm(t)$ denotes the Lebesgue--Stieltjes measure (which we identify with $|m'(t)|dt$ throughout). In particular, combining this with \eqref{hilbert} quickly leads to the inequality
\begin{equation}\label{basicBV}
\int_{\mathbb{R}}|T_mf|^2w\lesssim \int_{\mathbb{R}}|f|^2M^3w.
\end{equation}
Invoking classical weighted Littlewood--Paley theory for dyadic decompositions of the line (see \cite{W2} and \cite{BH} for further discussion) leads to the following weighted version of the Marcinkiewicz multiplier theorem (c.f. Kurtz \cite{Kurtz}).
\begin{theorem}\label{weightedmarcperez}
If $m:\mathbb{R}\rightarrow\mathbb{C}$ is a bounded function which is uniformly of bounded variation on dyadic intervals, that is
\begin{equation}\label{Marc}
\sup_{R>0}\int_{R\leq |\xi|\leq 2R}|m'(\xi)|d\xi<\infty,
\end{equation}
then
$$
\int_{\mathbb{R}}|T_mf|^2w\lesssim \int_{\mathbb{R}}|f|^2M^7w.
$$
\end{theorem}
The control of $m$ here by a power of the Hardy--Littlewood maximal operator is optimal in the sense that Theorem \ref{weightedmarcperez}, combined with the Hardy--Littlewood maximal theorem, implies the classical Marcinkiewicz multiplier theorem via \eqref{duality}. It would seem unlikely that the particular power of $M$ that features here is best-possible; here and throughout this paper we do not concern ourselves with such finer points.

Our goal is to establish versions of Theorem \ref{weightedmarcperez} which apply to much more singular (or ``rougher") multipliers.
A natural class of singular multipliers on the line, defined in terms of the so-called ``$r$-variation" was introduced by Coifman, Rubio de Francia and Semmes in \cite{CRdeFS}. For a function $m$ on an interval $[a,b]$ we define the $r$-variation of $m$ to be the supremum of the quantity
$$
\Bigl(\sum_{j=0}^{N-1}|m(x_{j+1})-m(x_j)|^r\Bigr)^{1/r}
$$
over all partitions $a=x_0<x_1<\cdots<x_N=b$ of $[a,b]$. We say that $m$ is a \emph{$V_r$ multiplier} if it has uniformly bounded $r$-variation on each dyadic interval. (Of course if $r=1$ this class reduces to the classical Marcinkiewicz multipliers.)
In \cite{CRdeFS} it is shown that if $m$ is a $V_r$ multiplier then $m$ is an $L^p(\mathbb{R})$ multiplier for $|1/p-1/2|<1/r$, considerably generalising the classical Marcinkiewicz multiplier theorem on the line.  With the possible exception of the endpoint, this result is sharp as may be seen from the specific multipliers
\begin{equation}\label{example1}
m_{\alpha,\beta}(\xi):=\frac{e^{i|\xi|^\alpha}}{(1+|\xi|^2)^{\beta/2}}; \;\;\;\alpha,\beta\geq 0,
\end{equation}
first studied by Hirschman \cite{Hirsch} (see \cite{S0} for further discussion). Indeed $m_{\alpha,\beta}$ is a $V_{r}$ multiplier if $\beta r=\alpha$, while being an $L^p$ multiplier if and only if $\alpha |1/p-1/2|\leq\beta$; see \cite{Hirsch}, \cite{Miy}. The endpoint case $|1/p-1/2|=1/r$ remains open in general for $V_r$ multipliers -- see \cite{TW} for further discussion and related results.

%For further discussion and perspectives on such $L^p$ multiplier results see \cite{TW} and \cite{Lacey}.

For the purposes of identifying \emph{optimal} controlling maximal operators we will confine attention to a subclass of the $V_r$ multipliers that retains some of the structure of the specific example \eqref{example1}.
Before we describe this subclass let us discuss some motivating examples.

The multiplier corresponding to the convolution kernel $e^{ix^\ell}$ appearing in \eqref{SamJon} coincides with the (generalised) Airy function
$$
Ai^{(\ell)}(\xi)=\int_{-\infty}^\infty e^{i(x^\ell+x\xi)}dx=c_0\frac{e^{ic_1|\xi|^{\frac{\ell}{\ell-1}}}+o(1)}{|\xi|^{\frac{\ell-2}{2(\ell-1)}}}
$$
as $|\xi|\rightarrow\infty$; here $c_0$ and $c_1$ are appropriate constants.
As standard Airy function asymptotics reveal, %(see for example \cite{Abrom} for the case $\ell=3$ at least),
the variation of this multiplier on dyadic intervals is unbounded. This multiplier, with its highly oscillatory behaviour as $|\xi|\rightarrow\infty$, belongs to a more general class of multipliers satisfying
\begin{equation}\label{airyinf}
m(\xi)=O(|\xi|^{-\beta}), \;\;\; m'(\xi)=O(|\xi|^{-\beta+\alpha-1})
\end{equation}
as $|\xi|\rightarrow\infty$. Here $\alpha,\beta\geq 0$, and of course the specific multiplier in \eqref{example1} is a model example.
In addition to multipliers whose derivatives can have strong singularities at \emph{infinity}, it is also natural to consider those which are singular at a \emph{point}. In particular, we might hope to control multipliers
%\begin{equation}\label{airyzero}
%m(\xi)=|\xi|^{-\beta}(e^{i|\xi|^{\alpha}}+o(1));\;\;\; |m'(\xi)|
%\end{equation}
satisfying \eqref{airyinf} as $|\xi|\rightarrow 0$ for $\alpha,\beta\leq 0$. Such singular multipliers, which were studied by Miyachi in \cite{Miy0} and \cite{Miy}, arise frequently in the study of oscillatory and oscillatory-singular integrals; see for example
\cite{S1}, \cite{Miy}, \cite{Sj} and \cite{CKS2}. See also \cite{Miy0} and
\cite{Miy} for a general $L^p(\mathbb{R})$ (and Hardy space $H^p(\mathbb{R})$) multiplier theorem under the
specific hypothesis \eqref{airyinf}. The following class of multipliers, which we denote $\mathcal{C}(\alpha,\beta)$, involves a Marcinkiewicz-type variation condition specifically designed to capture these Miyachi-type examples.
\subsection*{The class of multipliers}
%The Marcinkiewicz condition $\eqref{Marc}$ often fails in one of two ways: either the sequence
%$$
%\int_{R\leq |\xi|\leq 2R}|m'(\xi)|d\xi
%$$
%is unbounded as $R\rightarrow\infty$, or it is unbounded as $R\rightarrow 0$. The multiplier $\xi\mapsto e^{i|\xi|^\alpha}$ provides a very simple example of such failure for $\alpha>0$ and $\alpha<0$ respectively; indeed if one seeks uniformly bounded variation of this function, one is forced to decompose $[R, 2R]$ further into intervals of length $O(R^{-\alpha}R)$. Of course for $\alpha>0$ this leads to a finer (sub-dyadic) decomposition of $\{|\xi|\geq 1\}$, and for $\alpha<0$, a finer (sub-dyadic) decomposition of $\{|\xi|\leq 1\}$.
%These considerations motivate the following definition.
%These classes are modelled by the specific multipliers \eqref{airyinf}.
%\subsubsection*{The class $\mathcal{C}_\infty$}
For each $\alpha,\beta\in\mathbb{R}$ let $\mathcal{C}(\alpha,\beta)$ be the class of functions $m:\mathbb{R}\rightarrow\mathbb{C}$ for which
\begin{equation}\label{hyp0inf}
\supp(m)\subseteq \{\xi:|\xi|^\alpha\geq 1\},
\end{equation}
\begin{equation}\label{hyp1inf}
\sup_{\xi}\;|\xi|^\beta|m(\xi)|<\infty
\end{equation}
and
\begin{equation}\label{hyp2inf}
\sup_{R^\alpha\geq 1}\sup_{\substack{I\subseteq [R,2R]\\ \ell(I)=R^{-\alpha}R}}R^{\beta}\int_{\pm I}|m'(\xi)|d\xi<\infty.
\end{equation}
Here the supremum is taken over all subintervals $I$ of $[R,2R]$ of length $\ell(I)=R^{-\alpha}R$.
%The model example of such a multiplier is the function $m_{\alpha,\beta}$ defined in \eqref{example1}.
%Here the supremum is taken over all subintervals $I$ of $[R,2R]$ of length $\ell(I)=R^{-\alpha}R$.
%Again the model example of such a multiplier is the function $m_{\alpha,\beta}(\xi)\psi(\xi)$, where $\psi$ is a smooth cutoff and $m_{\alpha,\beta}$ is defined in \eqref{example1}.
\subsubsection*{Remarks}
\begin{itemize}
\item[(i)]
The support condition \eqref{hyp0inf} has no content for $\alpha=0$. For $\alpha>0$ and $\alpha<0$ it reduces to $\supp(m)\subseteq \{|\xi|\geq 1\}$ and $\supp(m)\subseteq \{|\xi|\leq 1\}$ respectively. A similar interpretation applies to the outermost supremum in \eqref{hyp2inf}.
\item[(ii)]
The case $\alpha=0$ is of course somewhat degenerate.
As is easily verified, the
class $\mathcal{C}(\alpha,\beta)$ reduces to the classical Marcinkiewicz multipliers when $\alpha=\beta=0$. Further, the fractional
integration multiplier $\xi\mapsto |\xi|^{-\beta}\in\mathcal{C}(0,\beta)$.
\item[(iii)] The model behaviour of a multiplier in $\mathcal{C}(\alpha,\beta)$ in the nondegenerate case $\alpha\not=0$ is that of the Miyachi multipliers \eqref{airyinf} as $|\xi|^\alpha\rightarrow\infty$.
% Here the asymptotic behaviour in \eqref{airyinf} is as $|\xi|^\alpha\rightarrow \infty$ if $\alpha>0$, and as $|\xi|\rightarrow 0$
%if $\alpha<0$.
\item[(iv)] An elementary calculation reveals that if $m\in\mathcal{C}(\alpha,\beta)$ then $m$ is a $V_r$ multiplier provided $\beta r=\alpha$.
We also note that the additional structure of the class $\mathcal{C}(\alpha,\beta)$ yields $L^p-L^q$ estimates
    for certain $q\not=p$ -- see the forthcoming Theorem \ref{mainmult}.
%\item[(iv)] The fractional integration multiplier $\xi\mapsto |\xi|^{-\beta}\in\mathcal{C}(0,\beta)$.
\item[(v)]
An elementary change of variables argument reveals that a multiplier $m\in\mathcal{C}(\alpha,\beta)$ if and only if $\widetilde{m}\in\mathcal{C}(-\alpha,-\beta)$, where $\widetilde{m}(\xi):=m(1/\xi)$.
The main point is that the diffeomorphism $\xi\mapsto 1/\xi:\mathbb{R}\backslash\{0\}\rightarrow \mathbb{R}\backslash\{0\}$ preserves dyadic intervals and (essentially) any lattice structure within them.
\item[(vi)] Unlike the $V_r$ multipliers, if $\alpha\not=0$ the class $\mathcal{C}(\alpha,\beta)$ is not dilation-invariant due to the distinguished role of the unit scale $R=1$. See the forthcoming Theorem \ref{main1q} for a natural dilation-invariant formulation.
\end{itemize}
We now introduce the family of maximal operators that will control these multipliers via \eqref{want}.
\subsection*{The controlling maximal operators}
For $\alpha,\beta\in\mathbb{R}$ we define the maximal operator $\mathcal{M}_{\alpha,\beta}$ by
\begin{equation}\label{maxdef}
\mathcal{M}_{\alpha, \beta}f(x)=\sup_{(r,y)\in \Gamma_{\alpha}(x)}\frac{r^{2\beta}}{r}\int_{|y-z|\leq r}f(z)dz
\end{equation}
where
\begin{equation}\label{regiondef}
\Gamma_\alpha(x)=\{(r,y):0<r^\alpha\leq 1\;\;\mbox{ and }\;\;|y-x|\leq r^{1-\alpha}\}.
\end{equation}
This family of maximal operators is of some independent interest.
When $\alpha=0$ the approach region $\Gamma_\alpha(x)$ is simply a cone with vertex $x$, and the associated maximal operator $\mathcal{M}_{\alpha, \beta}$ is equivalent to the classical fractional Hardy--Littlewood maximal operator
\begin{equation}\label{fractmax}
M_{2\beta} w(x):=\sup_{r>0}\frac{r^{2\beta}}{r}\int_{x-r}^{x+r}w.
\end{equation}
When $0< \alpha<1$ the maximal operators $\mathcal{M}_{\alpha,\beta}$ have also been considered before and originate in work of Nagel and Stein \cite{NS} on fractional maximal operators associated with more general nontangential approach regions.
However, as we have already mentioned, the above definitions also permit $\alpha\geq 1$ and $\alpha<0$, where one sees dramatic transitions in the nature of the region $\Gamma_\alpha$. In particular if $\alpha\geq 1$ then the situation is similar to that in \eqref{SamJonApp}, where tangential approach to infinite order is permitted; see \cite{BCSV} for the origins of such regions. Furthermore, for $\alpha<0$ we have
$$\Gamma_\alpha(x)=\{(r,y):r\geq 1\;\;\mbox{ and }\;\;|y-x|\leq r^{1-\alpha}\},$$
which may be viewed as an ``escape", rather than ``approach", region. Notice also that if $\beta<0$ we interpret $\mathcal{M}_{\alpha,\beta}$ as an \emph{improper}-fractional maximal operator.

The maximal operators $\mathcal{M}_{\alpha,\beta}$ are significant improvements on the controlling maximal operators $w\mapsto (Mw^s)^{1/s}$ that typically arise via classical $A_p$-weighted inequalities. Crudely estimating $\mathcal{M}_{\alpha,\beta}w$ pointwise using H\"older's inequality reveals that
\begin{equation}\label{crude}
\mathcal{M}_{\alpha,\beta}w\leq (M w^s)^{1/s}\;\;\mbox{ when }\;\; 2s\beta=\alpha.
\end{equation}
This allows the forthcoming Theorem \ref{main1} to be reconciled with certain $A_p$-weighted inequalities established by Chanillo, Kurt and Sampson in \cite{CKS}, \cite{CKS2}.
In Section \ref{appl} we provide necessary and sufficient conditions for $\mathcal{M}_{\alpha,\beta}$ to be bounded from $L^p$ to $L^q$. In particular, we see that $\mathcal{M}_{\alpha,\beta}$ is bounded on $L^s$ when $2s\beta=\alpha$; a property that does not follow from \eqref{crude}.

The main result of this paper is the following.
%Our main results are the following singular Marcinkiewicz-type multiplier theorems in a general weighted setting.
%For the sake of clarity of exposition we begin by presenting our results in one dimension.
\begin{theorem}\label{main1} Let $\alpha,\beta\in\mathbb{R}$. If $m\in\mathcal{C}(\alpha,\beta)$ then
\begin{equation}\label{maininequality1}
\int_{\mathbb{R}}|T_mf|^2w\lesssim\int_{\mathbb{R}}|f|^2M^6\mathcal{M}_{\alpha,\beta}M^4 w.
\end{equation}
\end{theorem}
It is interesting to contrast this result with the recent weighted variational Carleson theorem of Do and Lacey \cite{DoLacey}; see also \cite{EMTTW},
\cite{Lacey}.

As may be expected, the factors of Hardy--Littlewood maximal operator $M$ arising in Theorem \ref{main1} are of secondary importance, and to some extent occur for technical reasons. Since $M$ is bounded on $L^p$ for all $1<p\leq \infty$, the maximal operators $M^6\mathcal{M}_{\alpha,\beta}M^4$ and $\mathcal{M}_{\alpha,\beta}$ share the same $L^p-L^q$ bounds. The forthcoming Theorem \ref{mainmaximal} clarifies the $L^p-L^q$ behaviour of these operators.

It is perhaps helpful to make some further remarks about the nonsingular case $\alpha=0$ of the above theorem. As is immediately verified, the class of multipliers $\mathcal{C}(0,\beta)$ is precisely those satisfying the conditions
\begin{equation}\label{bdd}
\sup_{\xi\in\mathbb{R}}|\xi|^\beta|m(\xi)|<\infty
\end{equation}
and
\begin{equation}\label{hypnonsing}
\sup_{R>0}\;R^\beta\int_{R\leq |\xi|\leq 2R}|m'(\xi)|d\xi<\infty.
\end{equation}
For such ``classical" multipliers, Theorem \ref{main1} reduces to the weighted inequality
\begin{equation}\label{genperez}
\int_{\mathbb{R}}|T_mf|^2w\lesssim\int_{\mathbb{R}}|f|^2M^6 M_{2\beta}M^4 w,
\end{equation}
where $M_{2\beta}$ is the fractional Hardy--Littlewood maximal operator given by \eqref{fractmax}. When $\beta=0$ the conditions \eqref{bdd} and \eqref{hypnonsing} become those of the classical Marcinkiewicz multiplier theorem, and the resulting inequality \eqref{genperez} reduces -- up to factors of $M$ -- to the classical Theorem \ref{weightedmarcperez}. Noting that the multiplier $\xi\mapsto |\xi|^{-\beta}\in \mathcal{C}(0,\beta)$, again up to factors
%\footnote{Several of these factors of $M$ do not arise in the proof of this particularly simple case of Theorem \ref{main1}.}
of $M$ we recover the $1$-dimensional case of P\'erez's result in \cite{P2}.

Of course the class $\mathcal{C}(\alpha,\beta)$ is neither scale-invariant, nor facilitates quantification of the implicit constants in Theorem \ref{main1}. Our arguments, along with elementary scaling considerations, reveal the following.
\begin{theorem}\label{main1q}
Let $\alpha, \beta\in\mathbb{R}$ and $\lambda,C>0$. If $m:\mathbb{R}\rightarrow\mathbb{C}$ is such that
\begin{equation}\label{thm1q0}
\supp(m)\subseteq\{\xi:|\xi|^\alpha\geq \lambda^\alpha\},
\end{equation}
\begin{equation}\label{thm1q1}
\sup_{\xi}|\xi|^\beta |m(\xi)|\leq C
\end{equation}
and
\begin{equation}\label{thm1q2}
\sup_{R^\alpha\geq \lambda^\alpha}\sup_{\substack{I\subseteq [R,2R]\\\ell(I)=(R/\lambda)^{-\alpha}R}}R^\beta\int_{\pm I}|m'(\xi)|d\xi\leq C,
\end{equation}
then there exists an absolute constant $c>0$ such that
\begin{equation}\label{weightedquant}
\int_{\mathbb{R}}|T_mf|^2w\leq cC^2\int_{\mathbb{R}}|f|^2M^6\mathcal{M}_{\alpha,\beta}^{\lambda}M^4w,
\end{equation}
where
$$
\mathcal{M}_{\alpha,\beta}^{\lambda}w(x)=\sup_{(y,r)\in \Gamma_\alpha^{\lambda}(x)}\frac{r^{2\beta}}{r}\int_{y-r}^{y+r}w
$$
and
$$
\Gamma_\alpha^{\lambda}(x)=\{(y,r):0<r^\alpha\leq \lambda^{-\alpha}, \;\;\; |x-y|\leq \lambda^{-\alpha}r^{1-\alpha}\}.
$$
\end{theorem}
The hypotheses of Theorem \ref{main1q} are scale-invariant. More precisely, if $m$ satisfies \eqref{thm1q0}-\eqref{thm1q2} with parameter $\lambda=\eta$, then $\eta^\beta m(\eta\cdot)$ satisfies \eqref{thm1q0}-\eqref{thm1q2} with parameter $\lambda=1$.
%ALSO RECOVER SAMPSON ET AL VIA \eqref{crude}.
\subsubsection*{Organisation of the paper}
Our proof of Theorem \ref{main1} rests crucially on a certain Littlewood--Paley type square function estimate. This is presented in Section \ref{LPSection}. Section \ref{proofmain1} contains the proof of Theorem \ref{main1}, Section \ref{Sectionhigher} concerns extensions to higher dimensions, and finally Section \ref{boundingmax} is devoted to the $L^p-L^q$ boundedness properties of the maximal operators $\mathcal{M}_{\alpha,\beta}$.
We begin by presenting some applications and interpretations of Theorem \ref{main1}.
%\subsubsection*{Remark} It seems likely that one may formulate a version of Theorem \ref{main1} with H\"ormander-type (rather than Marcinkiewicz-type) hypotheses. We do not pursue this here.
\subsection*{Acknowledgments} We would like to thank Tony Carbery, Javi Duoandikoetxea, Sam Harrison, Luis Vega and Jim Wright for a number of helpful conversations on aspects of this paper.

\section{Applications and interpretations}\label{appl}
Here we present three distinct applications (or interpretations) of Theorem \ref{main1}.
\subsection{$L^p-L^q$ multipliers}
Our first application of Theorem \ref{main1} is to the theory of $L^p-L^q$ multipliers on the line. Such a multiplier theorem will follow from Theorem \ref{main1} via \eqref{duality} once we have suitable bounds on the maximal operators $\mathcal{M}_{\alpha,\beta}$.
\begin{theorem}\label{mainmaximal}
Let $1<p\leq q\leq\infty$ and $\alpha,\beta\in\mathbb{R}$. If $\alpha>0$ then $\mathcal{M}_{\alpha,\beta}$ is bounded from $L^p(\mathbb{R})$ to $L^q(\mathbb{R})$ if and only if
\begin{equation}\label{abpqpos}
\beta\geq\frac{\alpha}{2q}+\frac{1}{2}\Bigl(\frac{1}{p}-\frac{1}{q}\Bigr).
\end{equation}
If $\alpha=0$ then $\mathcal{M}_{\alpha,\beta}$ is bounded from $L^p(\mathbb{R})$ to $L^q(\mathbb{R})$ if and only if
\begin{equation}\label{abpqzero}
\beta=\frac{1}{2}\Bigl(\frac{1}{p}-\frac{1}{q}\Bigr).
\end{equation}
If $\alpha<0$ then $\mathcal{M}_{\alpha,\beta}$ is bounded from $L^p(\mathbb{R})$ to $L^q(\mathbb{R})$ if and only if
\begin{equation}\label{abpqneg}
\beta\leq\frac{\alpha}{2q}+\frac{1}{2}\Bigl(\frac{1}{p}-\frac{1}{q}\Bigr).
\end{equation}
\end{theorem}
\subsubsection*{Remarks}
When $\alpha=0$ Theorem \ref{mainmaximal} of course reduces to the well-known $L^p-L^q$ boundedness properties of the classical fractional Hardy--Littlewood maximal operator in one dimension -- see \cite{MW}. For $0\leq \alpha<1$ (the case of nontangential approach regions) and $p=q$, this result follows from the work of Nagel and Stein \cite{NS}. Certain particular cases of Theorem \ref{mainmaximal} in the region $\alpha>1$ are established in \cite{BH}, following arguments in \cite{BCSV}.
Our proof, which extends further the arguments in \cite{BCSV}, follows by establishing a corresponding endpoint Hardy space result when $p=1$ -- see Section \ref{boundingmax}.

Combining Theorems \ref{main1} and \ref{mainmaximal} yields the following unweighted Marcinkiewicz-type multiplier theorem.
\begin{corollary}\label{mainmult}
Let $2\leq p\leq q<\infty$, $\alpha,\beta\in\mathbb{R}$ and suppose $m\in\mathcal{C}(\alpha,\beta)$. If $\alpha>0$ and $$\beta\geq\alpha\left(\frac{1}{2}-\frac{1}{p}\right)+\frac{1}{p}-\frac{1}{q},$$ or $\alpha=0$ and $$\beta=\frac{1}{p}-\frac{1}{q},$$
or $\alpha< 0$ and $$\beta\leq\alpha\left(\frac{1}{2}-\frac{1}{p}\right)+\frac{1}{p}-\frac{1}{q},$$ then $m$ is an $L^p(\mathbb{R})-L^q(\mathbb{R})$ multiplier.
\end{corollary}
\subsubsection*{Remarks}
Corollary \ref{mainmult}, which modestly generalises a number of well-known results, is optimal subject to the (inevitable) constraint $p,q\geq 2$ -- see \cite{Miy0} and \cite{Miy}.
% IS THIS TRUE? IS THIS DEALT WITH IN \cite{Miy0}?
%For $\alpha>0$ with $\alpha\not=1$, this optimality is most easily seen via the analysis of the specific examples \eqref{example1} in Miyachi
%\cite{Miy}.
However, as the examples in \cite{Miy} and \cite{Miy0} also suggest, unless $p=q$, Corollary \ref{mainmult} is unlikely to lead to optimal results in the full range $1\leq p,q\leq\infty$.
If $\alpha\not=0$ then, by duality and interpolation we may conclude that $m$ is an $L^p(\mathbb{R})$ multiplier for all $1<p<\infty$ satisfying the familiar condition $|1/2-1/p|\leq\beta/\alpha$. This generalises the $L^p$ (as opposed to $H^p$) multiplier results of Miyachi \cite{Miy0} in dimension $n=1$.
%This sharpens and extends the known results for the specific multipliers $m_{\alpha,\beta}$ given in \eqref{example1}.
If $\alpha=0$ then Corollary \ref{mainmult} reduces to the classical one-dimensional Marcinkiewicz
multiplier theorem on setting $p=q$, since $m$ is a Marcinkiewicz multiplier if and only if $m\in\mathcal{C}(0,0)$. The special case $\alpha=0$ also generalises the classical Hardy--Littlewood--Sobolev theorem on fractional integration since the multiplier $|\xi|^{-\beta}$ belongs to $\mathcal{C}(0,\beta)$.

\subsection{Oscillatory convolution kernels on the line}\label{subsect}
%\textbf{PUT REFERENCE TO \cite{CKS2} HERE TOO.}
The method of stationary phase permits Theorem \ref{main1} to be applied to a variety of explicit oscillatory convolution operators on the line.
For example, for $a>0$ with $a\not=1$, and $1-a/2\leq b<1$, consider the convolution kernel $K_{a,b}:\mathbb{R}
\backslash\{0\}\rightarrow\mathbb{C}$ given by
$$K_{a,b}(x)=\frac{e^{i|x|^a}}{(1+|x|)^b}.$$
%Here $\eta\in C^\infty(\mathbb{R})$ is supported away from the origin and satisfies $\eta(x)=1$ for $|x|\geq 1$.
The corresponding convolution operator is well-understood on $L^p$, with
\begin{equation}\label{Sjolinoperators}
\|K_{a,b}*f\|_{p}\lesssim\|f\|_p\;\;\;\iff\;\;\;p_0\leq p\leq p_0',
\end{equation}
where $p_0=\frac{a}{a+b-1}$; see \cite{Sj}, \cite{JS}. As we shall see, an application of Theorem \ref{main1} quickly leads to the following.
\begin{theorem}\label{applSj}
If $a>0$ with $a\not=1$ and $1-a/2\leq b<1$ then
\begin{equation}\label{used}\int_{\mathbb{R}}|K_{a,b}*f|^2w\lesssim
\int_{\mathbb{R}}|f|^2M^6\mathcal{M}_{\alpha,\beta}M^4w
\end{equation}
where $\alpha=\frac{a}{a-1}$ and $\beta=\frac{a/2+b-1}{a-1}$.
\end{theorem}
This theorem is optimal in the sense that it allows us to recover \eqref{Sjolinoperators} (and indeed more general $L^p-L^q$ estimates) from Theorem \ref{mainmaximal} via
\eqref{duality}. Notice that if $0<a<1$ then $\alpha:=\frac{a}{a-1}<0$, and so the controlling maximal
operator $\mathcal{M}_{\alpha,\beta}$ corresponds to an \textit{escape} region. Similarly, if $a>1$ then $\alpha>0$ and so
$\mathcal{M}_{\alpha,\beta}$ corresponds to an \textit{approach} region. Theorem \ref{applSj} may of course be viewed as a generalisation (modulo factors of $M$) of the inequality \eqref{SamJon}.

In order to deduce Theorem \ref{applSj} from Theorem \ref{main1} we simply observe that, up to a couple of well-behaved ``error" terms, the multiplier
$\widehat{K}_{a,b}$ belongs to $\mathcal{C}(\alpha,\beta)$. Let us begin by handling the portion of $K_{a,b}$ in a neighbourhood of the origin (where the kernel lacks smoothness). Let $\eta\in C_c^\infty(\mathbb{R})$ be such that $\eta(x)=1$ for $|x|\leq 1$, and write $K_{a,b}=K_{a,b,0}+K_{a,b,\infty}$, where $K_{a,b,0}=\eta K_{a,b}$. Since $K_{a,b,0}$ is rapidly decreasing,
by the Cauchy--Schwarz inequality we have
$$\int_{\mathbb{R}}|K_{a,b,0}*f|^2w\leq\|K_{a,b,0}\|_1\int_{\mathbb{R}}|f|^2|K_{a,b,0}|*w\lesssim
\int_{\mathbb{R}}|f|^2M^{1}w,$$
where
$$M^1w(x):=\sup_{r\geq 1}\frac{1}{2r}\int_{x-r}^{x+r}w.$$
The claimed inequality \eqref{used} for the portion of the kernel $K_{a,b,0}$ now follows from the elementary pointwise bound
$$M^1w\lesssim AM^1w\leq \mathcal{M}_{\alpha,\beta}M^1w\leq \mathcal{M}_{\alpha,\beta}Mw\leq M^6\mathcal{M}_{\alpha,\beta}M^4w,$$
where the averaging operator $A$ is given by
$$Aw(x)=\frac{1}{2}\int_{x-1}^{x+1}w.$$
It thus remains to prove \eqref{used} for the portion $K_{a,b,\infty}$.
In order to force the support hypothesis \eqref{hyp0inf} we introduce a function $\mathcal{\psi}\in C^\infty(\mathbb{R})$ such that $\psi(\xi)=0$ when $|\xi|^\alpha\leq 1$ and $\psi(\xi)=1$ when $|\xi|^\alpha\geq 2$. Writing $m_0=(1-\psi)
\widehat{K}_{a,b,\infty}$ and $m_1=\psi\widehat{K}_{a,b,\infty}$, it suffices to show that
\begin{equation}
\label{nuff}\int_{\mathbb{R}}|T_{m_j}f|^2w\lesssim\int_{\mathbb{R}}|f|^2M^6\mathcal{M}_{\alpha,\beta}M^4w
\end{equation}
for $j=0,1$. A standard stationary phase argument (see \cite{Sj} for explicit details) reveals that
$m_1$ satisfies the Miyachi-type bounds \eqref{airyinf} as $|\xi|^\alpha\rightarrow\infty$. Hence $m_1\in\mathcal{C}(\alpha,\beta)$, yielding \eqref{nuff} for $j=1$ by Theorem \ref{main1}.
The multiplier $m_0$ is less interesting, being the Fourier transform of a rapidly decreasing function (again, see \cite{Sj} for further details). Arguing as we did for the portion $K_{a,b,0}$ establishes \eqref{nuff} for $j=0$, completing the proof.
%For example, consider the so-called hyper-singular convolution kernel $K_{a,b}(x)=e^{i|x|^{-a}}|x|^{-b}\phi(x)$, where $a,b>0$ and $\phi\in C_c^\infty(\mathbb{R})$. This kernel was first studied by Hirschman \cite{Hirsch}. For $b\geq 1$ this kernel fails to be integrable at the origin, although provided the oscillatory character $a$ is sufficiently large, it is well-known that $f\mapsto K_{a,b}*f$ extends to a bounded operator from $L^p$ to $L^q$ if and only if ... A rudimentary stationary phase argument yields
%$$m(\xi):=\widehat{K}_{a,b}(\xi)=c_1|\xi|^{-\beta}e^{i c_2|\xi|^\alpha}+O(|\xi|^{-\beta-\alpha/2})\;\;\mbox{ as }\;\;|\xi|\rightarrow\infty,$$ where $\alpha=a/(a+1)$ and $\beta=(a-2b+2)/(2a+2b)$; see \cite{S1}. A further stationary phase argument reveals that
%$$
%|m'(\xi)|\lesssim |\xi|^{-\beta+\alpha-1},\;\;\;\; \textbf{CHECK}$$
%which is easily seen to be sufficient for $m$ to belong to the class $\mathcal{C}_\infty(\alpha,\beta)$. Thus
%\begin{equation}\label{hypersing}
%\int_{\mathbb{R}}|K_{a,b}*f|^2w\lesssim\int_{\mathbb{R}}|f|^2M^6\mathcal{M}_{\alpha,\beta}M^4w
%\end{equation}
%by Theorem \ref{main1}.

For a more far-reaching discussion relating to the asymptotics of Fourier transforms of oscillatory kernels, see \cite{S1}, and what Stein refers to as the ``duality of phases".
\subsection{Spatial regularity of solutions of dispersive equations}
Theorem \ref{main1} has an interesting interpretation in the context of spatial regularity of solutions to dispersive equations. For example, applying\footnote{Strictly speaking we are applying Theorem 2 to a portion of the multiplier supported away from the origin, and dealing with the portion near the origin by other (elementary) means. See Section \ref{subsect} for further details.} Theorem 2 to the multiplier $m_{2,\beta}$ given by \eqref{example1} yields
$$
\int_{\mathbb{R}}|e^{i\partial^2}f|^2w\lesssim\int_{\mathbb{R}}|(I-\partial^2)^{\beta/2}f|^2M^6\mathcal{M}_{2,\beta}M^4w$$
for all $\beta\geq 0$. Using the scale-invariant inequality \eqref{weightedquant} with $\lambda=t^{-1/2}$, a similar statement may be made for the operator
$e^{it\partial^2}$; namely
$$
\int_{\mathbb{R}}|e^{it\partial^2}f|^2w\lesssim\int_{\mathbb{R}}|(t^{-1}I-\partial^2)^{\beta/2}f|^2M^6\mathcal{M}_{2,\beta}^{t^{-1/2}}M^4w,$$
with implicit constant independent of $t>0$. It is perhaps more natural to rewrite this as
$$
\int_{\mathbb{R}}|e^{it\partial^2}f|^2w\lesssim\int_{\mathbb{R}}|(I-t\partial^2)^{\beta/2}f|^2M^6\mathfrak{M}_tM^4w,
$$
where
$$
\mathfrak{M}_tw(x):=\sup_{(y,r)\in\Lambda(x)}r^{2\beta}\frac{1}{t^{1/2}r}\int_{y-t^{1/2}r}^{y+t^{1/2}r}w
$$
and
$$
\Lambda(x)=\{(y,r): 0<r\leq 1,\;\;\; |x-y|\leq t^{1/2}/r\},
$$
so that the degeneracy as $t\rightarrow 0$ is more apparent.
The resulting $L^p$ multiplier theorem at $t=1$ (see Corollary \ref{mainmult} in the case $q=p$) is the inequality
$$
\|e^{i\partial^2}f\|_{L^p(\mathbb{R})}\lesssim\|f\|_{W^{\beta,p}}
$$
for $\beta\geq 2|1/2-1/p|$. Here $W^{\beta,p}$ denotes the classical inhomogeneous $L^p$ Sobolev space. This optimal Sobolev inequality, which goes back to Miyachi \cite{Miy}, describes the regularity loss in $L^p(\mathbb{R})$ for a solution to the Schr\"odinger equation with initial data in $L^p(\mathbb{R})$. Naturally this interpretation applies equally well to the wave, Airy and more
general (pseudo) differential dispersive equations. Similar conclusions, for the Schr\"odinger equation at least, may be reached in higher dimensions using the results of Section \ref{higherdimsect}; see also \cite{Miy}.

%For oscillatory kernels with more general phase functions, the heuristic principle that Stein \cite{S1} refers to as the ``duality of phases"

%The method of stationary phase reveals that convolution kernels of the form $K(x)=e^{ix^a}$ have Fourier multipliers belonging to the class $\mathcal{C}_\infty$. PUT HERE STRONGLY SINGULAR CONVOLUTION OPERATORS AND AIRY FUNCTION MULTIPLIERS... DO WE ALSO GET SOME OSCILLATORY SINGULAR INTEGRALS - \textbf{LOOK AT MULTIPLIER ASYMPTOTICS FOR VARIOUS CONVOLUTION KERNELS!}

\section{Weighted inequalities for a lattice square function}\label{LPSection}
%PUT HERE THE MAIN RESULT IN ALL DIMENSIONS RELATING TO THE FORWARD AND REVERSE SQUARE FUNCTION ESTIMATES.
In this section we present the forward and reverse weighted Littlewood--Paley square function estimates that underpin our proof of Theorem \ref{main1}. We formulate our results in $\mathbb{R}^n$ in anticipation of higher dimensional applications in Section
\ref{Sectionhigher}.

Let $\Psi\in\mathcal{S}(\mathbb{R}^n)$ be such that $\supp(\widehat{\Psi})\subseteq [-1,1]^n$ and
$$
\sum_{k\in\mathbb{Z}^n}\widehat{\Psi}(\xi-k)=1$$
for all $\xi\in\mathbb{R}^n$. Such a function may of course be constructed by defining $\widehat{\Psi}=\chi_{[-1/2,1/2]^n}*\phi$, for a function $\phi\in C_c^\infty(\mathbb{R}^n)$ of suitably small support and integral $1$.

For each $t\in (0,\infty)^n$ we define the $n\times n$ dilation matrix $\delta(t):=\diag(t_1,\cdots, t_n)$, and the rectangular box
$B(t):=\delta(t)^{-1}[-1,1]^n=[-1/t_1,1/t_1]\times\cdots\times [-1/t_n,1/t_n]$.

Now let $R'\in (0,\infty)^n$ and decompose $\mathbb{R}^n$ into a lattice of rectangles $\{\rho_k\}$ as follows. For each $k\in\mathbb{Z}^n$ let
$$\rho_k=\delta(R')(\{k\}+[-\tfrac{1}{2},\tfrac{1}{2}]^n),$$
making $\rho_k$ the axis-parallel rectangular cell centred at $\delta(R')k=(R_1'k_1,\hdots,R_n'k_n)$ with $j$th side-length $R_j'$. Defining $\Psi_k\in\mathcal{S}(\mathbb{R}^n)$ by
$$\widehat{\Psi}_k(\xi)=\widehat{\Psi}(\delta(R')^{-1}\xi -k),$$
we have
\begin{equation}\label{unity}
\sum_{k\in\mathbb{R}^n}\widehat{\Psi}_k\equiv 1
\end{equation}
and $$\supp(\widehat{\Psi}_k)\subseteq\widetilde{\rho}_k$$ for each $k\in\mathbb{Z}^n$.
Here $\widetilde{\rho}_k$ denotes the concentric double of $\rho_k$.
Finally, let the operator $S_{k}$ be given by $\widehat{S_{k}f}=\widehat{\Psi}_{k}\widehat{f}$.

For the operators $S_{k}$ we have the following essentially standard square function estimate. Very similar results may be found in several places in the literature, including \cite{cordo}, \cite{Rubio} and \cite{BCSV}.
\begin{proposition}\label{classicalLP}
\begin{equation}\label{classineq}\int_{\mathbb{R}^n}\sum_{k}|S_{k}f|^2w\lesssim\int_{\mathbb{R}^n}|f|^2M_Sw\end{equation}
uniformly in $R'$, where $M_S$ denotes the strong maximal function.
\end{proposition}
A reverse weighted inequality, where the function $f$ is controlled by the square function $(\sum_{k}|S_{k}f|^2)^{1/2}$, is rather more subtle, and is the main content of this section.
\begin{theorem}\label{mainLP} Suppose $R\in (0,\infty)^n$ is such that $R_j\geq R_j'$ for each $1\leq j\leq n$, and let $\rho$ be an axis-parallel rectangle in $\mathbb{R}^n$ of $j$th side-length $R_j$. If $\supp(\widehat{f}\;)\subseteq\rho$ then
$$\int_{\mathbb{R}^n}|f|^2w\lesssim\int_{\mathbb{R}^n}\sum_{k}|S_{k}f|^2M_SA_{R,R'}M_Sw,$$
where the operator $A_{R,R'}$ is given by
$$
A_{R,R'}w(x)=\sup_{y\in\{x\}+B(R')}\frac{1}{|B(R)|}\int_{\{y\}+B(R)}w.
$$
%Here $B(t):=\delta(t)^{-1}[-1,1]^n=[-1/t_1,1/t_1]\times\cdots\times [-1/t_n,1/t_n]$ for each $t\in (0,\infty)^n$.
%$$
%A_{R,R'}w(x)=\sup_{\substack{|y_1-x_1|\leq 1/R_1'\\\vdots\\
%|y_n-x_n|\leq 1/R_n'}} R_1\cdots R_n\int_{\substack{|z_1-y_1|\leq 1/R_1\\\vdots\\|z_n-y_n|\leq 1/R_n}}w(z)dz.
%$$
\end{theorem}
\subsubsection*{Remark} As the following proof reveals, Theorem \ref{mainLP} continues to hold if the operators $S_k$ are replaced by the genuine frequency-projection operators defined by $\widehat{S_kf}=\chi_{\rho_k}\widehat{f}$.
\subsection*{Proof of Theorem \ref{mainLP}}
We begin by exploiting the Fourier support hypothesis on $f$ to mollify the weight $w$.
Let $\Phi\in\mathcal{S}(\mathbb{R}^n)$ be such that $\widehat{\Phi}=1$ on $[-1,1]^n$. Observe that if we define $\Phi_{R}\in\mathcal{S}(\mathbb{R}^n)$ by
$\widehat{\Phi}_R(\xi)=\widehat{\Phi}(\delta(R)^{-1}\xi)=\widehat{\Phi}(\xi_1/R_1,\hdots,\xi_n/R_n)$, then $f=f*(M_{\xi_\rho}\Phi_R)$. Here $M_{\xi_\rho}\Phi_R(x)=e^{-2\pi ix\cdot\xi_\rho}\Phi_R(x)$ and $\xi_\rho$ denotes the centre of $\rho$. A standard application of the Cauchy--Schwarz inequality and Fubini's theorem reveals that
\begin{equation}\label{mollification}
\int_{\mathbb{R}^n}|f|^2w=\int_{\mathbb{R}^n}|f*(M_{\xi_\rho}\Phi_R)|^2w\leq \|\Phi_R\|_1\int_{\mathbb{R}^n}|f|^2|\Phi_R|*w\lesssim \int_{\mathbb{R}^n}|f|^2|\Phi_R|*w.
\end{equation}
The final inequality here follows since the functions $\Phi_R$ are normalised in $L^1$.

Now, by \eqref{unity} we have $$f=\sum_{k}S_{k}f.$$
This raises issues of orthogonality for the operators $S_{k}$ on $L^2(w_1)$. Although the weight $w_1=|\Phi_{R}|*w$ is smooth, in order for us to have any (almost) orthogonality we should expect to need an improved smoothness consistent with a mollification by $|\Phi_{R'}|$ rather than $|\Phi_R|$. We thus seek an efficient way of dominating $w_1$ by such an improved weight\footnote{This idea is somewhat reminiscent of the classical fact that if $(Mw^s)^{1/s}<\infty$ a.e. and $s>1$ then $w\leq (Mw^s)^{1/s}\in A_1\subset A_2$; see the discussion following \eqref{hilbertprim}.}. This ingredient, which is based on an argument in \cite{BCSV}, comes in two simple steps. First define the weight function $w_2$ by
$$
w_2(x)=\sup_{y\in\{x\}+B(R')}\; w_1(y).
$$
Certainly $w_2$ dominates $w_1$ pointwise, although $w_2$ will not in general be sufficiently smooth for our purposes. Let $\Theta\in\mathcal{S}(\mathbb{R}^n)$ be a nonnegative function whose Fourier transform is nonnegative and compactly supported, and let $$w_3=\Theta_{R'}*w_2,$$ where $\Theta_{R'}$ is defined by $\widehat{\Theta}_{R'}(\xi)=\widehat{\Theta}(\delta(R')^{-1}\xi)=\widehat{\Theta}(\xi_1/R_1',\hdots,\xi_n/R_n')$. By construction $w_3$ has Fourier support in $\{\xi:|\xi_j|\lesssim R_j', \; 1\leq j\leq n\}$, and so by Parseval's theorem we have the desired almost orthogonality:%\footnote{There is a minor technicality that we are glossing over here. Strictly speaking it is necessary to dilate $\rho_k$ and $\rho_{k'}$ by a slightly larger absolute constant factor in order to guarantee \eqref{orth}.}
\begin{eqnarray}\label{orth}
\langle S_{k}f,S_{k'}f\rangle_{L^2(w_3)}=0\;\;\mbox{ if }\;\; |k-k'|\gtrsim 1.
%\widetilde{\rho}_k\cap\widetilde{\rho}_{k'}=\emptyset.
\end{eqnarray}
Despite its improved smoothness, this new weight $w_3$ continues to dominate $w_1$.
\begin{lemma}\label{stilldom}
$w_2\lesssim w_3$.
\end{lemma}
\begin{proof}
By dilating $\Theta$ by an absolute constant if necessary, we may assume that $\Theta\gtrsim 1$ on $[-1,1]^n$.
Consequently,
$$
w_3(0)\gtrsim \frac{1}{|B(R')|}\int_{B(R')}w_2(x)dx.
$$
Now let $B_1,B_2,\hdots, B_{2^n}$ be the intersections of $B(R')$ with the $2^n$ coordinate hyperoctants of $\mathbb{R}^n$. It will suffice to show that there exists $\ell\in \{1,2,\hdots,2^n\}$ such that $w_2(x)\geq w_2(0)$ for all $x\in B_\ell$. To see this we suppose, for a contradiction, that there exist $x_\ell\in B_\ell$ such that $w_2(x_\ell)<w_2(0)$ for each $1\leq\ell\leq 2^n$. Thus, by the definition of $w_2$ we have
$$\sup_{x\in \{x_\ell\}+B(R')}w_1(x)<w_2(0)\;\mbox{ for }\;1\leq \ell\leq 2^n.$$ However, since $$B(R')\subseteq\bigcup_{\ell=1}^{2^n}
(\{x_\ell\}+B(R')),$$
$\sup_{x\in B(R')}w_1(x)<w_2(0)$, contradicting the definition of $w_2(0)$.
\end{proof}
Combining \eqref{mollification}, Lemma \ref{stilldom} and the orthogonality property \eqref{orth} we obtain
\begin{equation}\label{mostdone}
\int_{\mathbb{R}^n}|f|^2w\lesssim \int_{\mathbb{R}^n}\sum_{k}|S_{k}f|^2 w_3.
\end{equation}
In order to complete the proof of Theorem \ref{mainLP} it remains to show that $w_3(x)\lesssim M_SA_{R,R'}M_Sw(x)$ uniformly in $x$ and $R, R'$. Since $w_3(x)\lesssim M_Sw_2(x)$ it suffices to show that $w_2(x)\lesssim A_{R,R'}M_Sw(x)$. Further, by translation-invariance, it is enough to deal with the case $x=0$. To see this we define the maximal operator $M_S^{(R)}$ by
$$M_S^{(R)}w(y)=\sup_{r\geq 1}\frac{1}{|rB(R)|}\int_{\{y\}+rB(R)}w.$$ Notice that $M_S^{(R)}w\leq M_Sw$. Using the rapid decay of $\Phi$ and elementary considerations we have
$$w_1(y)=|\Phi_{R}|*w(y)\lesssim M_S^{(R)}w(y)\lesssim \frac{1}{|B(R)|}\int_{\{y\}+B(R)}M_S^{(R)}w,$$ and so
$$
w_2(0)\lesssim\sup_{y\in B(R')}\frac{1}{|B(R)|}\int_{\{y\}+B(R)}M_Sw= A_{R,R'}M_Sw(0)$$ uniformly in $R, R'$, as required.

%For completeness we include a proof of the more classical Proposition \ref{classicalLP}.
%\subsection*{Proof of Proposition \ref{classicalLP}}
%By the dilation-invariance of the inequality \eqref{classineq}, it is enough to deal with the isotropic case where $R_1'=\cdots=R_n'=1$.

\section{The proof of Theorem \ref{main1}}\label{proofmain1}
The proof we present combines the essential ingredients of the standard proof of the Marcinkiewicz multiplier theorem (see for example \cite{S0} or \cite{Duo}) and the square function estimates from Section \ref{LPSection}.

By standard weighted Littlewood--Paley theory (see \cite{BH} for further details) it suffices to prove that
\begin{equation}\label{enough}
\int_{\mathbb{R}}|T_mf|^2w\lesssim\int_{\mathbb{R}}|f|^2M^5\mathcal{M}_{\alpha,\beta}M w,
\end{equation}
holds for functions $f$ with Fourier support in the dyadic interval $\pm[R,2R]$, with bounds uniform in $R^\alpha\geq 1$.

Suppose that $\supp(\widehat{f})\subseteq \pm[R,2R]$ for some $R^\alpha\geq 1$. We begin by applying Theorem \ref{mainLP} with $n=1$, $R'=R^{-\alpha}R$ and $\rho=\pm[R,2R]$. For each $k\in\mathbb{Z}$ let $\rho_k$, $\widetilde{\rho}_k$, $\Psi_{k}$ and $S_{k}$ be as in Section \ref{LPSection}. By Theorem \ref{mainLP} we have
\begin{equation}\label{usingmainLP}
\int_{\mathbb{R}}|T_mf|^2w\lesssim\int_{\mathbb{R}}\sum_{k}|S_{k}T_mf|^2MA_{R,R'}Mw
\end{equation}
uniformly in $R^\alpha\geq 1$. Of course the case $R=1$ (as with the case $\alpha=0$) is somewhat degenerate here, although we note that the conclusion \eqref{usingmainLP} does retain some content.

Next we invoke the standard representation formula
\begin{equation}\label{keyobs}
S_{k}T_mf(x)=m(a_{k})S_{k}f(x)+\int_{\widetilde{\rho}_k}U_\xi S_{k}f(x)m'(\xi)d\xi,
\end{equation}
where $a_{k}=\inf\:\widetilde{\rho}_k$ and $U_\xi$ is defined by \begin{equation}\label{uu}\widehat{U_\xi f}=\chi_{[\xi,\infty)}\widehat{f}.\end{equation}
In order to see \eqref{keyobs}, which is a minor variant of \eqref{marcrep}, we use the Fourier inversion formula to write
\begin{eqnarray*}
\begin{aligned}
S_{k}T_mf(x)&=\int_{\widetilde{\rho}_k}e^{ix\xi}\widehat{\Psi}_{k}(\xi)m(\xi)\widehat{f}(\xi)d\xi\\
&=-\int_{\widetilde{\rho}_k}\frac{\partial}{\partial\xi}\Bigl(\int_{\xi}^\infty\widehat{\Psi}_{k}(t)\widehat{f}(t)e^{ixt}dt\Bigr)m(\xi)d\xi\\
&=m(a_{k})S_{k}f(x)+\int_{\widetilde{\rho}_k}\Bigl(\int_{\mathbb{R}}\chi_{[\xi,\infty)}(t)\widehat{\Psi}_{k}(t)\widehat{f}(t)e^{ixt}dt\Bigr)m'(\xi)d\xi\\
&=m(a_{k})S_{k}f(x)+\int_{\widetilde{\rho}_k}U_\xi S_{k}f(x)m'(\xi)d\xi.
\end{aligned}
\end{eqnarray*}
Applying Minkowski's inequality to \eqref{keyobs} we obtain
\begin{eqnarray*}
\begin{aligned}
\Bigl(\int_{\mathbb{R}}|S_kT_mf|^2MA_{R,R'}Mw\Bigr)^{1/2}&\leq |m(a_k)|\Bigl(\int_{\mathbb{R}}|S_kf|^2MA_{R,R'}Mw\Bigr)^{1/2}\\
&+\int_{\widetilde{\rho}_k}\Bigl(\int_{\mathbb{R}}|U_\xi S_kf|^2MA_{R,R'}Mw\Bigr)^{1/2}|m'(\xi)|d\xi.
\end{aligned}
\end{eqnarray*}
Since $U_\xi=\frac{1}{2}(I+iM_{-\xi}HM_\xi)$ where $M_\xi f(x):=e^{-2\pi ix\xi}f(x)$ and $H$ is the Hilbert transform, an application of \eqref{hilbert} yields
$$
\int_{\mathbb{R}}|U_\xi S_kf|^2MA_{R,R'}Mw\lesssim \int_{\mathbb{R}}|S_kf|^2M^4A_{R,R'}Mw
$$
uniformly in $\xi$, $k$ and $R$. Using this along with the hypotheses \eqref{hyp1inf} and \eqref{hyp2inf} yields
$$
\int_{\mathbb{R}}|S_k T_m f|^2MA_{R,R'}Mw\lesssim R^{-2\beta}\int_{\mathbb{R}}|S_kf|^2M^4A_{R,R'}Mw
$$
uniformly in $k$ and $R$. Here we have used the fact that $|a_k|\sim R$. Thus by \eqref{usingmainLP} and Proposition \ref{classicalLP} we have
$$
\int_{\mathbb{R}}|T_m f|^2w\lesssim R^{-2\beta}\int_{\mathbb{R}}|f|^2M^5A_{R,R'}Mw
$$
uniformly in $R^\alpha\geq 1$.
Inequality \eqref{enough} now follows from the elementary observation that $$R^{-2\beta}A_{R,R^{-\alpha}R}w(x)\lesssim \mathcal{M}_{\alpha,\beta}w(x)$$ uniformly in $x$ and $R^\alpha\geq 1$.

\section{Extensions to higher dimensions}\label{higherdimsect}\label{Sectionhigher}
%THERE WOULD APPEAR TO BE SEVERAL DIFFERENT RESULTS THAT ONE COULD PROVE HERE...
Theorem \ref{main1} has a natural generalisation to higher dimensions. It should be pointed out that this generalisation, being of Marcinkiewicz type in formulation, is not motivated by multipliers of the form \eqref{airyinf}, but rather by tensor products of such one-dimensional multipliers.
%\begin{equation}\label{example3}
%\frac{e^{i|\xi_1|^{\alpha_1}}}{\langle \xi_1\rangle^{\beta_1}}\frac{e^{i|\xi_2|^{\alpha_2}}}{\langle \xi_2\rangle^{\beta_2}}.
%\end{equation}
For the sake of simplicity we confine our attention to two dimensions. Just as with the classical Marcinkiewicz multiplier theorem, this is already typical of the general situation.

For $\alpha,\beta\in\mathbb{R}^2$ let $\mathcal{C}(\alpha,\beta)$ denote the class of functions $m:\mathbb{R}^2\rightarrow\mathbb{C}$
for which
\begin{equation}\label{hyp20}
\supp(m)\subseteq\{\xi\in\mathbb{R}^2: |\xi_1|^{\alpha_1}\geq 1,\;\;|\xi_2|^{\alpha_2}\geq 1\},
\end{equation}
\begin{equation}\label{hyp21}
\sup_{\xi_2}\sup_{\xi_1}\;|\xi_2|^{\beta_2}|\xi_1|^{\beta_1}|m(\xi_1,\xi_2)|<\infty,
\end{equation}
\begin{equation}\label{hyp22}
\sup_{\xi_2}|\xi_2|^{\beta_2}\Biggl\{\sup_{R_1^{\alpha_1}\geq 1}\sup_{\substack{I_1\subseteq [R_1,2R_1]\\ \ell(I_1)=R_1^{-\alpha_1}R_1}}
R_1^{\beta_1}\int_{\pm I_1}\Bigl|\frac{\partial m}{\partial \xi_1}\Bigr|d\xi_1\Biggr\}<\infty,
\end{equation}
\begin{equation}\label{hyp23}
\sup_{\xi_1}|\xi_1|^{\beta_1}\Biggl\{\sup_{R_2^{\alpha_2}\geq 1}\sup_{\substack{I_2\subseteq [R_2,2R_2]\\ \ell(I_2)=R_2^{-\alpha_2}R_2}}
R_2^{\beta_2}\int_{\pm I_2}\Bigl|\frac{\partial m}{\partial \xi_2}\Bigr|d\xi_2\Biggr\}<\infty,
\end{equation}
and
\begin{equation}\label{hyp24}
\sup_{R_2^{\alpha_2}\geq 1}\sup_{\substack{I_2\subseteq [R_2,2R_2]\\ \ell(I_2)=R_2^{-\alpha_2}R_2}}\sup_{R_1^{\alpha_1}\geq 1}\sup_{\substack{I_1\subseteq [R_1,2R_1]\\ \ell(I_1)=R_1^{-\alpha_1}R_1}}R_2^{\beta_2}R_1^{\beta_1}\int_{\pm I_2}\int_{\pm I_1}\Bigl|\frac{\partial^2 m}{\partial \xi_1 \partial \xi_2}\Bigr|d\xi_1d\xi_2<\infty.
\end{equation}
Although these conditions might appear rather complicated, it is straightforward to verify that the tensor product $\mathcal{C}(\alpha_1,\beta_1)\otimes\mathcal{C}(\alpha_2,\beta_2)\subset\mathcal{C}(\alpha,\beta)$, and that $\mathcal{C}(0,0)$ is precisely the classical Marcinkiewicz multipliers on $\mathbb{R}^2$.
\begin{theorem}\label{main3}
\label{mainhigher} If $m\in\mathcal{C}(\alpha,\beta)$ then
\begin{equation}\label{maininequalityhigher}
\int_{\mathbb{R}^2}|T_mf|^2w\lesssim\int_{\mathbb{R}^2}|f|^2M_S^{9}\mathcal{M}_{\alpha,\beta}M_S^{7} w,
\end{equation}
where
\begin{equation}\label{maxdef}
\mathcal{M}_{\alpha,\beta}w(x)=\sup_{(r_1,y_1)\in \Gamma_{\alpha_1}(x_1)}\;\sup_{(r_2,y_2)\in \Gamma_{\alpha_2}(x_2)}\frac{r_1^{2\beta_1}}{r_1}\frac{r_2^{2\beta_2}}{r_2}\int_{|y_1-z_1|\leq r_1}\int_{|y_2-z_2|\leq r_2}w(z)dz
\end{equation}
and $M_S$ denotes the strong maximal function.
\end{theorem}
\subsection*{The proof of Theorem \ref{main3}}
The proof we present is very similar to the one-dimensional case. By standard weighted Littlewood--Paley theory (again, see \cite{BH} for details) it suffices to prove that
\begin{equation}\label{enoughhigher}
\int_{\mathbb{R}^2}|T_mf|^2w\lesssim\int_{\mathbb{R}^2}|f|^2M_S^{8}\mathcal{M}_{\alpha,\beta}M_S w,
\end{equation}
holds for functions $f$ with Fourier support in $(\pm[R_1,2R_1])\times(\pm [R_2,2R_2])$, with bounds uniform in $R_1^{\alpha_1},
R_2^{\alpha_2}\geq 1$.

Assuming such a restriction we apply Theorem \ref{mainLP} with $n=2$, $R'=(R_1^{-\alpha_1}R_1,R_2^{-\alpha_2}R_2)$ and $\rho=(\pm[R_1,2R_1])\times(\pm [R_2,2R_2])$. For each $k\in\mathbb{Z}^2$ let $\rho_k$, $\widetilde{\rho}_k$, $\Psi_k$ and $S_k$ be as in Section \ref{LPSection}. By Theorem \ref{mainLP} we have
\begin{equation}\label{usingmainLPhigher}
\int_{\mathbb{R}^2}|T_mf|^2w\lesssim\int_{\mathbb{R}^2}\sum_{k\in\mathbb{Z}^2}|S_kT_mf|^2M_SA_{R,R'}M_Sw
\end{equation}
uniformly in $R$.

In what follows $\pi_1,\pi_2:\mathbb{R}^2\rightarrow\mathbb{R}$ denote the coordinate projections $\pi_1x=x_1$ and $\pi_2x=x_2$, and for each $k$ we define $a_k\in\mathbb{R}^2$ by
$a_k=(\inf \pi_1\widetilde{\rho}_k, \inf \pi_2\widetilde{\rho}_k)$, making $a_k$ the bottom left vertex of the axis-parallel rectangle $\widetilde{\rho}_k$.

Now, taking our cue again from the standard proof of the classical Marcinkiewicz multiplier theorem we write
\begin{eqnarray}\label{keyobs2}
\begin{aligned}
S_{k}T_mf(x)&=m(a_k)S_{k}f(x)\\&+\int_{\pi_1\widetilde{\rho}_k}U_{\xi_1}^{(1)}S_{k}f(x)\frac{\partial m}{\partial\xi_1}(\xi_1,\pi_2a_k)d\xi_1\\&+
\int_{\pi_2\widetilde{\rho}_k}U_{\xi_2}^{(2)}S_{k}f(x)\frac{\partial m}{\partial\xi_2}(\pi_1a_k,\xi_2)d\xi_2\\&+
\int_{\widetilde{\rho}_k}U_{\xi_2}^{(2)}U_{\xi_1}^{(1)}S_{k}f(x)
\frac{\partial^2 m}{\partial\xi_1\partial\xi_2}(\xi_1,\xi_2)d\xi_1 d\xi_2,
\end{aligned}
\end{eqnarray}
where $U_{\xi_j}^{(j)}$ denotes the operator $U_{\xi_j}$, defined in \eqref{uu}, acting in the $j$th variable.
Applying Minkowski's inequality we obtain
\begin{eqnarray*}
\begin{aligned}
\Bigl(\int_{\mathbb{R}^2}&|S_kT_mf|^2M_SA_{R,R'}M_Sw\Bigr)^{1/2}\\&\leq |m(a_k)|\Bigl(\int_{\mathbb{R}^2}|S_kf|^2M_SA_{R,R'}M_Sw\Bigr)^{1/2}\\
&+\int_{\pi_1\widetilde{\rho}_k}\Bigl(\int_{\mathbb{R}^2}|U_{\xi_1}^{(1)}S_kf|^2M_SA_{R,R'}M_Sw\Bigr)^{1/2}\Bigl|\frac{\partial m}{\partial{\xi_1}}(\xi_1,\pi_2a_k)\Bigr|d\xi_1\\
&+\int_{\pi_2\widetilde{\rho}_k}\Bigl(\int_{\mathbb{R}^2}|U_{\xi_2}^{(2)}S_kf|^2M_SA_{R,R'}M_Sw\Bigr)^{1/2}\Bigl|\frac{\partial m}{\partial{\xi_2}}(\pi_1a_k,\xi_2)\Bigr|d\xi_2\\
&+\int_{\widetilde{\rho}_k}\Bigl(\int_{\mathbb{R}^2}|U_{\xi_2}^{(2)} U_{\xi_1}^{(1)}S_kf|^2M_SA_{R,R'}M_Sw\Bigr)^{1/2}\Bigl|\frac{\partial^2 m}{\partial{\xi_1}\partial{\xi_2}}(\xi_1,\xi_2)\Bigr|d\xi\\
&=I+II+III+IV.
\end{aligned}
\end{eqnarray*}
For $I$ we use the facts that $|\pi_1a_k|\sim R_1$ and $|\pi_2a_k|\sim R_2$, along with \eqref{hyp21} to obtain
$$I\lesssim R_2^{-\beta_2}R_1^{-\beta_1}\Bigl(\int_{\mathbb{R}^2}|S_kf|^2M_SA_{R,R'}M_Sw\Bigr)^{1/2}$$ uniformly in $k$.
For $II$, following the proof of Theorem \ref{main1}, we apply \eqref{hilbert} in the first variable to obtain
$$II\lesssim \int_{\pi_1\widetilde{\rho}_k}\Bigl(\int_{\mathbb{R}^2}|S_kf|^2M_S^4A_{R,R'}M_Sw\Bigr)^{1/2}\Bigl|\frac{\partial m}{\partial{\xi_1}}(\xi_1,\pi_2a_k)\Bigr|d\xi_1,$$
which by \eqref{hyp22} yields
$$
II\lesssim R_2^{-\beta_2}R_1^{-\beta_1}\Bigl(\int_{\mathbb{R}^2}|S_kf|^2M_S^4A_{R,R'}M_Sw\Bigr)^{1/2}$$
uniformly in $k$.
By \eqref{hyp23} and symmetry it follows that $III$ satisfies the same bound. The final term $IV$ is potentially the most interesting as it involves using a weighted bound on the double Hilbert transform. By a two-fold application of \eqref{hilbert}, followed by \eqref{hyp24}, we obtain
$$IV\lesssim R_2^{-\beta_2}R_1^{-\beta_1}\Bigl(\int_{\mathbb{R}^2}|S_kf|^2M_S^7A_{R,R'}M_Sw\Bigr)^{1/2}.$$
Thus by \eqref{usingmainLPhigher} and Proposition \ref{classicalLP} we have
$$\int_{\mathbb{R}^2}|T_mf|^2w\lesssim R_2^{-2\beta_2}R_1^{-2\beta_1}\int_{\mathbb{R}^2}|S_kf|^2M_S^8A_{R,R'}M_Sw$$
uniformly in $R_1^{\alpha_1}, R_2^{\alpha_2}\geq 1$. Inequality \eqref{enoughhigher} now follows on observing that
$$
R_2^{-2\beta_2}R_1^{-2\beta_1}A_{R,R'}w(x)\lesssim\mathcal{M}_{\alpha,\beta}w(x)$$
uniformly in $x$ and $R_1^{\alpha_1},R_2^{\alpha_2}\geq 1$.
\subsubsection*{Remarks}
%\begin{itemize}
%\item[(i)]
The above arguments raise certain basic questions about weighted inequalities for various multiparameter operators in harmonic analysis. For instance, for which powers $k\in\mathbb{N}$ do we have
$$
\int_{\mathbb{R}^n}|Tf|^2w\lesssim\int_{\mathbb{R}^n}|f|^2M_S^kw
$$
for classical product Calder\'on--Zygmund operators $T$ on $\mathbb{R}^n$ with $n\geq 2$? As we have seen, crudely applying the one dimensional result of Wilson \cite{W} separately in each variable allows us to take $k=3n$. Reducing this power would of course lead to a reduction in the number of factors of $M_S$ in the statement of Theorem \ref{mainhigher}.

As we have already discussed, since Theorem \ref{mainhigher} involves Marcinkiewicz-type hypotheses it really belongs to the ``multiparameter" theory of multipliers. It is conceivable that a variant may be obtained involving a H\"ormander-type hypothesis on sub-lacunary annuli in $\mathbb{R}^n$; that is, involving hypotheses on quantities of the form
$$
\int_{R_j\leq |\xi|<R_{j+1}}\Bigl|\Bigl(\frac{\partial}{\partial\xi}\Bigr)^\gamma m(\xi)\Bigr|^2d\xi
$$
for certain sub-lacunary sequences $(R_j)$ and multi-indices $\gamma$. A very general result of this type (which might permit the
radii $(R_j)$ to accumulate away from zero) is likely to be difficult as it would naturally apply to the Bochner--Riesz multipliers. There are of course many other conditions that one might impose, from the above all the way down to the higher dimensional analogue of the Miyachi condition \eqref{airyinf} in \cite{Miy} and \cite{Miy0}; see also \cite{Carb83}.
%\end{itemize}

\section{Proof of Theorem \ref{mainmaximal}}\label{boundingmax} In this section we give a proof of Theorem \ref{mainmaximal}.
Our argument is a generalisation of those in \cite{BCSV} and \cite{BH}; see also \cite{NS}.
As the case $\alpha=0$ reduces to the $L^p-L^q$ boundedness of the classical fractional Hardy--Littlewood maximal function, we may assume that $\alpha\not=0$.

The claimed necessity of the conditions \eqref{abpqpos}, \eqref{abpqzero} and \eqref{abpqneg} follows from testing the putative $L^p-L^q$ bound for $\mathcal{M}_{\alpha,\beta}$ on the characteristic function $f_{\nu}=\chi_{[-\nu,\nu]}$. The necessary conditions follow by taking limits as both $\nu\rightarrow 0$ and $\nu\rightarrow\infty$. We leave these elementary calculations to the reader.

It will suffice to establish the $L^p-L^q$ boundedness of $\mathcal{M}_{\alpha,\beta}$ for exponents $1<p\leq q\leq\infty$ on the sharp line
\begin{equation}\label{abpq}
\beta=\frac{\alpha}{2q}+\frac{1}{2}\Bigl(\frac{1}{p}-\frac{1}{q}\Bigr).
\end{equation}
As our proof of Theorem \ref{mainmaximal} rests on a Hardy space estimate, it is necessary to regularise the averaging in the definition of $\mathcal{M}_{\alpha,\beta}$. To this end let
$P$ be a nonnegative compactly supported bump function which is positive on $[-1,1]$, let $P_r(x)=r^{-1}P(x/r)$, and define the maximal operator $\widetilde{\mathcal{M}}_{\alpha,\beta}$ by
$$\widetilde{\mathcal{M}}_{\alpha,\beta}w(x)=\sup_{(y,r)\in\Gamma_\alpha(x)}r^{2\beta}|P_r*w(y)|.$$ Since $\mathcal{M}_{\alpha,\beta}w\lesssim\widetilde{\mathcal{M}}_{\alpha,\beta}w$ pointwise uniformly, it suffices to prove that $\widetilde{\mathcal{M}}_{\alpha,\beta}$ is bounded from $L^p(\mathbb{R})$ to $L^q(\mathbb{R})$ when \eqref{abpq} holds. Since $\widetilde{\mathcal{M}}_{\alpha,0}$ is bounded on $L^\infty(\mathbb{R})$, and $\widetilde{\mathcal{M}}_{\alpha,1/2}$ is bounded from $L^1(\mathbb{R})$ to $L^\infty(\mathbb{R})$, by analytic interpolation (see \cite{S0}) it suffices to prove that $\widetilde{\mathcal{M}}_{\alpha,\alpha/2}$ is bounded from $H^1(\mathbb{R})$ to $L^1(\mathbb{R})$. We establish this by showing that
\begin{equation}\label{atombound}\|\widetilde{\mathcal{M}}_{\alpha,\alpha/2}a\|_1\lesssim 1
\end{equation}
uniformly in $H^1$-atoms $a$. By translation-invariance we may suppose that the support interval $I$ of $a$ is centred at the origin.
Our estimates will be based on the standard and elementary pointwise bound
\[ |P_r*a(x)| \lesssim \left\{ \begin{array}{lll}
         1/|I| & \mbox{if $r\leq |I|$, $|x|\leq 5|I|/2$};\\
        |I|/r^2 & \mbox{if $r\geq |I|$, $|x|\leq 5r/2$};\\
    0 & \mbox{otherwise,}
    \end{array} \right. \]
which follows from the smoothness of $P$ and the mean value zero property of $a$.
As the nature of $\Gamma_\alpha$ is fundamentally different in the cases $\alpha<0$, $0<\alpha\leq 1$ and $\alpha>1$, we divide the analysis into three cases. For $\alpha<0$ and $\alpha>0$ the interesting situation is when $|I|\gtrsim 1$ and $|I|\lesssim 1$ respectively.
\subsubsection*{Case 1: $\alpha<0$}
Elementary considerations reveal that if $|I|\lesssim 1$ then
\[ \widetilde{\mathcal{M}}_{\alpha,\alpha/2}a(x) \lesssim \left\{ \begin{array}{ll}
         |I| & \mbox{if $|x|\lesssim 1$};\\
    |I|/|x|^{-(2-\alpha)/(1-\alpha)} & \mbox{otherwise,}
    \end{array} \right. \]
and if $|I|\gtrsim 1$ then
\[\widetilde{\mathcal{M}}_{\alpha,\alpha/2}a(x) \lesssim \left\{ \begin{array}{ll}
         |I|^{-1}|x|^{\alpha/(1-\alpha)} & \mbox{if $|x|\lesssim|I|^{1-\alpha}$};\\
        |I||x|^{-(2-\alpha)/(1-\alpha)} & \mbox{otherwise}.
    \end{array} \right. \]
In both cases \eqref{atombound} follows by direct calculation.
\subsubsection*{Case 2: $0<\alpha\leq 1$}
For technical reasons it is convenient to deal first with the particularly simple case $\alpha=1$. If
$|I|\gtrsim 1$ then arguing similarly we obtain
\[ \widetilde{\mathcal{M}}_{1,1/2}a(x) \lesssim \left\{ \begin{array}{ll}
         |I|^{-1} & \mbox{if $|x|\lesssim |I|$};\\
    0 & \mbox{otherwise,}
    \end{array} \right. \]
and if $|I|\lesssim 1$ then
\[\widetilde{\mathcal{M}}_{1,1/2}a(x) \lesssim \left\{ \begin{array}{ll}
         1 & \mbox{if $|x|\lesssim 1$};\\
    0 & \mbox{otherwise.}
    \end{array} \right. \]
Clearly in both cases \eqref{atombound} follows immediately.

Suppose now that $0<\alpha<1$. If $|I|\gtrsim 1$ then
\[ \widetilde{\mathcal{M}}_{\alpha,\alpha/2}a(x) \lesssim \left\{ \begin{array}{ll}
         |I|^{-1} & \mbox{if $|x|\lesssim |I|$};\\
    0 & \mbox{otherwise,}
    \end{array} \right. \]
and if $|I|\lesssim 1$ then
\[\widetilde{\mathcal{M}}_{\alpha,\alpha/2}a(x) \lesssim \left\{ \begin{array}{lll}
         |I|^{-(1-\alpha)} & \mbox{if $|x|\lesssim|I|^{1-\alpha}$};\\
        |I||x|^{-(2-\alpha)/(1-\alpha)} & \mbox{if $|I|^{1-\alpha}\lesssim |x|\lesssim 1$};\\
    0 & \mbox{otherwise.}
    \end{array} \right. \]
Again, in both cases \eqref{atombound} follows directly.

\subsubsection*{Case 3: $\alpha>1$} If $|I|\gtrsim 1$ then
\[ \widetilde{\mathcal{M}}_{\alpha,\alpha/2}a(x) \lesssim \left\{ \begin{array}{ll}
         |I|^{-1} & \mbox{if $|x|\lesssim |I|$};\\
    |I|^{-1}|x|^{\alpha/(1-\alpha)} & \mbox{otherwise,}
    \end{array} \right. \]
and if $|I|\lesssim 1$ then
\[\widetilde{\mathcal{M}}_{\alpha,\alpha/2}a(x) \lesssim \left\{ \begin{array}{ll}
         |I|^{-(1-\alpha)} & \mbox{if $|x|\lesssim|I|^{1-\alpha}$};\\
        |I|^{-1}|x|^{\alpha/(1-\alpha)} & \mbox{if $|I|^{1-\alpha}\lesssim |x|$}.\\
    \end{array} \right. \]
Once again \eqref{atombound} follows.
%\section{Further results} OBVIOUS HIGHER DIMENSIONAL ANALOGUES OF THE MAXIMAL OPERATORS.

\bibliographystyle{amsplain}

\end{document}